\documentclass[12pt,oneside,reqno]{amsart}
\usepackage{graphicx}
\usepackage{indentfirst,csquotes}

\usepackage[margin=3cm]{geometry} 
\usepackage{stackengine} 
\usepackage{caption} 
\usepackage{bm} 

\usepackage{amssymb,amsthm,amsmath}
\usepackage{xcolor,paralist,hyperref,
fancyhdr,etoolbox}

\newtheorem{theorem}{Theorem}[section] 
\newtheorem{thmx}{Theorem}

\newtheorem{lemma}[theorem]{Lemma}
\newtheorem{proposition}[theorem]{Proposition}
\newtheorem{corollary}[theorem]{Corollary}

\theoremstyle{definition} 
\newtheorem{example}[theorem]{Example}
\newtheorem{definition}[theorem]{Definition}

\numberwithin{equation}{section} 

\hypersetup{ colorlinks=true, linkcolor=black, filecolor=black, urlcolor=black }

\usepackage{mathrsfs}
\usepackage{mathtools}
\usepackage{mathdots}
\usepackage{stmaryrd}
\usepackage[all,cmtip]{xy}
\usepackage{stackengine} 
\newcommand{\Mod}[1]{\ (\mathrm{mod}\ #1)}
\usepackage{stackengine} 

\usepackage{float}

\newtheorem{remark}[theorem]{Remark}

\def\N{\operatorname{\mathbb{N}}}
\def\Z{\operatorname{\mathbb{Z}}}
\def\Q{\operatorname{\mathbb{Q}}}
\def\R{\operatorname{\mathbb{R}}}
\def\C{\operatorname{\mathbb{C}}}
\def\F{\operatorname{\mathbb{F}}}
\def\P{\operatorname{\mathbb{P}}}
\def\G{\operatorname{\mathbb{G}}}

\def\GL{\operatorname{GL}}
\def\SO{\operatorname{SO}}
\def\SL{\operatorname{SL}}
\def\Sp{\operatorname{Sp}}
\def\Hom{\operatorname{Hom}}
\def\spec{\operatorname{Spec}}
\def\pic{\operatorname{Pic}}
\def\gr{\operatorname{Gr}}
\def\fpt{\operatorname{fpt}}
\def\lct{\operatorname{lct}}
\def\Fl{\operatorname{Fl}}

\begin{document}

\title{Computing the {$F$}-pure Threshold of Flag Varieties} 
\author{Justin Fong}
\date{}
\address{Justin Fong, Department of Mathematics, Purdue University, West Lafayette, IN 47907, USA}
\email{jafong1@gmail.com}

\let\thefootnote\relax

\begin{abstract}
We compute the $F$-pure threshold of the natural cone over flag varieties in characteristic $p>0$. Our calculations are mainly focused on flag varieties that are arithmetically Gorenstein, but we offer some results in the non-Gorenstein case. Our goal is to determine the $a$-invariant of the cone. As a result, the $F$-pure thresholds we find are independent of the characteristic $p$, hence one immediately gets the value of the log canonical threshold of flags in characteristic 0 as well. 
\end{abstract} 

\maketitle

\section{Introduction}

Throughout this paper, all rings $(R,\mathfrak{m},\Bbbk)$ are standard graded domains over a field $\Bbbk$, where $\mathfrak{m}=R_{>0}$ is the irrelevant maximal ideal, and $\Bbbk$ is an algebraically closed field, unless otherwise stated.

The $F$-pure threshold, introduced by Takagi and Watanabe in \cite{MR2097584},
is a numerical invariant associated to rings with $F$-pure singularities. $F$-pure singularities are an example class of $F$-singularities, which are classes of singularities defined in positive characteristic using the Frobenius map. The $F$-pure threshold is defined as follows. Suppose $R$ has characteristic $p$ and is $F$-finite and $F$-pure. Let $\mathfrak{a}\subseteq R$ be a nonzero ideal. For a real number $t\in\R_{\ge0}$, the pair $(R,\mathfrak{a}^t)$ is called \emph{$F$-pure} if for every $q=p^e\gg0$, there exists an element $f\in \mathfrak{a}^{\lfloor (q-1)t \rfloor}$ such that the map $R\to R^{1/q}$ sending 1 to $f^{1/q}$ splits as an $R$-module homomorphism. The \emph{$F$-pure threshold} of $\mathfrak{a}$ is defined as
\[\fpt(\mathfrak{a}) := \sup\{t\in\R_{\ge0} \mid (R,\mathfrak{a}^t) \ \text{is $F$-pure}\}.\]
When $\mathfrak{a}=\mathfrak{m}$, we denote $\fpt(R):=\fpt(\mathfrak{m})$, and call this the \emph{$F$-pure threshold of $R$}. 

The invariant $\fpt(\mathfrak{a})$ can be thought of as the characteristic $p$ counterpart of the \emph{log canonical threshold} $\lct(\mathfrak{a})$ in characteristic zero, which is a measurement of singularities on the variety $V(\mathfrak{a})$. Although their definitions differ, many important connections have been established between these two invariants, one of which is that the log canonical threshold is approximated by the $F$-pure threshold through reduction modulo $p$ (see \cite[Theorem 3.4]{MR2185754}).

Explicit formulas for $\fpt(R)$ have been computed in \cite{MR2097584}, \cite{MR2549545}, \cite{MR3354064}, \cite{MR3719471}, and \cite{MR4809892}. The aim of this paper is to add to this list a formula for $\fpt(R)$ where $R$ is the homogeneous coordinate ring of a \emph{flag variety} under the Pl\"{u}cker embedding, which is a natural embedding into projective space. Although we study various types of flag varieties determined by semisimple algebraic groups, the common feature these rings share is the Gorenstein property. Thus, by a well-known fact in the theory of $F$-pure thresholds, $\fpt(R)$ is equal to the negative of the $a$-invariant (see Theorem \ref{thm: fpt graded rings}), so the problem reduces to determining this number. We are able to calculate the $a$-invariant of the coordinate rings of their associated flag varieties via the combinatorics of an underlying poset, and through root systems of the associated semisimple group acting on the flag variety. Our main results are listed in the following.

\begin{thmx}\label{thm: summary of results}
Let $R$ be the homogeneous coordinate ring of the flag variety $X$, under the Pl\"{u}cker embedding, of the following types:
    \begin{enumerate}
        \item $X$ is the partial flag variety $\Fl(d_1,\dots,d_r)$. Then $\fpt(R) = n+d_r-d_1$ (see Corollary \ref{cor: fpt of flag}). \label{item: 1}

        \item $X$ is a minucule Grassmannian. The values of $\fpt(R)$ are given by Table \ref{table: minuscule} (see page \pageref{table: minuscule}). \label{item: 2}
        
        \item $X$ is a Grassmannian of exceptional type. The values of $\fpt(R)$ are given by Table \ref{table: exceptional} (see page \pageref{table: exceptional}). \label{item: 3}
        
        \item $X$ is a generalized flag $G/Q$ that is projectively embedded via certain dominant weights. The values of $\fpt(R)$ are given by Propositions \ref{prop: non-gorenstein flag} and \ref{prop: non-gorenstein grassmannians}. \label{item: 4}
    \end{enumerate}
\end{thmx}

Another main result of our paper is that we make a connection between the $a$-invariant of the homogeneous coordinate ring $R$ of a general Grassmannian $G/P$ and the root system associated to the algebraic group $G$. This is stated as follows.

\begin{thmx}[Corollary \ref{cor: a-invariant via root systems}]\label{thm: root systems}
Let $\Phi$ be the root system associated to $G$, and $P$
be a maximal parabolic subgroup determined by the set of simple roots $I=\Delta\setminus\{\alpha\}$, for some simple root $\alpha\in\Delta$. Let $\varpi$ be the fundamental weight corresponding to $P$. Then
\[-a(R)\varpi = \sum_{\alpha\in\Phi^+\setminus\Phi_I}\alpha.\]
\end{thmx}

This paper is organized as follows. In Section \ref{sec: Main Tools}, we list the main tools used to calculate the $F$-pure threshold of our rings. In Section \ref{sec: classical flags}, we describe the coordinate rings $R$ of classical flag varieties and calculate their $a$-invariant and $F$-pure threshold via a certain chain in the underlying poset of $R$, proving (\ref{item: 1}) of Theorem \ref{thm: summary of results}. In Section \ref{sec: Generalized Flag Varieties}, we focus on general flag varieties defined by algebraic groups and compute the values stated in (\ref{item: 2}), (\ref{item: 3}), \& (\ref{item: 4}) of Theorem \ref{thm: summary of results}. The values in (\ref{item: 4}) are computed using properties of Veronese subrings, while the values in (\ref{item: 2}) are obtained using the same combinatorial ideas used in Section \ref{sec: classical flags}. We also prove that general Grassmannian varieties under the Pl\"{u}cker embedding have Gorenstein coordinate rings (see Proposition \ref{prop: Grassmannian Gorenstein a-invariant}), and from the proof of this fact we are able to derive Theorem \ref{thm: root systems}, which we use to calculate the values in (\ref{item: 3}) of Theorem \ref{thm: summary of results}. 

As a final remark, all values of the $F$-pure threshold found in this article are independent of the field characteristic $p$, hence we obtain the corresponding log canonical thresholds as well, via Theorem \ref{thm: lct limit of fpt} below.

\subsection*{Acknowledgements}
I would like to thank my advisor, Uli Walther, for suggesting me this problem, and for providing helpful feedback on this paper. I also want to thank Daniel Le, whom I consulted for the algebraic group aspects of the paper, and the referee for carefully proofreading the article.  

This work was supported by NSF-grant DMS-2100288 and by Simons Foundation Collaboration Grant for Mathematicians \#580839.


\section{Background}\label{sec: Main Tools}

\emph{In this paper, all rings are assumed to have the standard grading}.

The \emph{$a$-invariant} of a positively graded Cohen-Macaulay $\Bbbk$-algebra $R$ with canonical module $\omega_R$ is defined to be
\[a(R) = -\min\{j : (\omega_R)_j \ne 0\}.\]
The $F$-pure threshold of a Gorenstein ring is known to be the following. 

\begin{theorem}\cite[Theorem 5.2]{MR3778235}\label{thm: fpt graded rings}
Let $(R,\mathfrak{m},\Bbbk)$ be a standard graded $F$-finite, $F$-pure Gorenstein $\Bbbk$-algebra. Then 
\begin{equation}
    \fpt(R)=-a(R).
\end{equation}
\end{theorem}

The $F$-pure threshold can be used to approximate the log canonical threshold in the following way.

\begin{theorem}\cite[Theorem 3.4]{MR2185754}\label{thm: lct limit of fpt}
Let $\mathfrak{a}$ be a homogeneous ideal of $\Z[x_1,\dots,x_n]$. For each prime $p$, let $\mathfrak{a}_p$ denote the image of $\mathfrak{a}$ in $\F_p[x_1,\dots,x_n]$. Then
\begin{equation}
\lct(\mathfrak{a}_{\C})=\lim_{p\to\infty}\fpt(\mathfrak{a}_p),
\end{equation}
where $\mathfrak{a}_{\C}:=\mathfrak{a}\cdot\C[x_1,\dots,x_n]$.
\end{theorem}

\subsection{Algebra with Straightening Laws}

We give the definition of a algebra with straightening laws as found in \cite{MR680936} and \cite{MR986492}. However, we use the term algebra with straightening law to refer to what \cite{MR680936} calls the ordinal Hodge algebra.

\begin{definition}
Let $B$ be a commutative ring, and $A$ be a $B$-algebra where there is a finite subset $\Pi\subset A$ with partial order $\le$ (so $\Pi$ is a poset). Then $A$ is a graded \emph{algebra with straightening laws} (ASL for short) on $\Pi$ over $B$ if the following conditions are satisfied:
\begin{enumerate}
    \item [(1)] $A$ is a positively graded $B$-algebra, and $\Pi$ consists of homogeneous elements of positive degree which generate $A$ over $B$.

    \item[(2)] The set of \emph{standard monomials} $\{\xi_1\cdots\xi_r \mid \text{$\xi_i\in\Pi$ and $\xi_1\le\cdots\le\xi_r$},\ r\in\N\}$ are linearly independent over $B$.

    \item[(3)] For every pair of incomparable elements $\delta,\tau$ in $\Pi$, the product $\delta\tau$ is a linear combination of standard monomials
    \[\delta\tau = \sum_ib_i\xi_{i1}\cdots\xi_{ir_i}, \quad b_i\in B\setminus\{0\}, \quad \xi_{i1}\le\cdots\le\xi_{ir_i},\]
    with $\xi_{i1}<\tau$ and $\xi_{i1}<\delta$ for all $i$. These expressions are called the \emph{straightening relations}.
\end{enumerate}
\end{definition}

Next, we discuss a closed formula for the $a$-invariant of a graded ASL $R$ on a poset $\Pi$ over a field $\Bbbk$, which was first given by Bruns and Herzog in \cite{MR1188581}. In this paper, it is enough to assume $\Pi$ is a distributive lattice. See \cite{MR1188581} for the general conditions on $\Pi$. Given $x,y\in\Pi$, we say $y$ is a \emph{cover} of $x$ if $x<y$, and there is no $z\in\Pi$ such that $x<z<y$. The \emph{principal chain} of a lattice $\Pi$ is a chain $\mathcal{P}(\Pi)=\xi_1,\dots,\xi_m$ formed inductively as follows, $\xi_1$ is the unique minimal element of $\Pi$. If $\xi_1,\xi_2,\dots,\xi_i$ is defined and $\xi_i$ is not the unique maximal element of $\Pi$,
\[\xi_{i+1}=\nu_1\sqcup\cdots\sqcup\nu_v,\]
where $\nu_1,\dots,\nu_v$ are elements which cover $\xi_i$, and $\sqcup$ denotes the join in $\Pi$. If $\xi_i$ is the maximal element of $\Pi$, then the inductive construction terminates and $\xi_1,\dots,\xi_i$ is the principal chain. By using the principal chain, the $a$-invariant of an ASL is given as follows.

\begin{theorem}\cite[Theorem 1.1]{MR1188581}\label{thm: ASL a-invariant}
Let $R$ be a graded ASL over $\Bbbk$ on a distributive lattice $\Pi$ with principal chain $\mathcal{P}(\Pi)=\xi_1,\dots,\xi_m$. Then its $a$-invariant is given by
\begin{equation}\label{equ: a-invariant of schubert cycle}
    a(R) = -\sum_{i=1}^m\deg\xi_i.
\end{equation}
\end{theorem}

\begin{remark}\label{rem: principal chain standard grading}
Since we assume the rings in this paper have the standard grading, each $\deg\xi_i=1$, therefore (\ref{equ: a-invariant of schubert cycle}) becomes
\[a(R) = -m = -|\mathcal{P}(\Pi)|,\]
where $|\mathcal{P}(\Pi)|$ is the cardinality of the principal chain of $\Pi$.  
\end{remark}

\section{Flag Varieties in a Vector Space}\label{sec: classical flags}

Let $V$ be an $n$-dimensional vector space over a field $\Bbbk$. A \emph{flag} in $V$ is a sequence $(V_1,\dots,V_r)$ of proper subspaces of $V$ such that $V_1\subset\cdots\subset V_r$. If $d_i=\dim_\Bbbk V_i$, then $1\le d_1<\cdots<d_r<n$, and we say the flag $(V_1,\dots,V_r)$ has signature $(d_1,\dots,d_r)$.
If each $\dim_{\Bbbk}V_i=i$, we say $(V_1,\dots,V_r)$ is a \emph{complete flag}, otherwise it is called a \emph{partial flag}. The \emph{partial flag variety} $\Fl(d_1,\dots,d_r)$ is defined as the space of all partial flags of signature $(d_1,\dots,d_r)$, while the space of all complete flags, $\Fl(1,2,\dots,n-1)$, is called the \emph{full flag variety}. When $r=1$ and $d=d_1$, the flag variety becomes the space of all $d$-dimensional subspaces of $V$, called the \emph{Grassmannian variety}, denoted $\gr(d,V)$.

\subsection{The Coordinate Ring of a Partial Flag Variety}
We describe the multi-homogeneous coordinate ring of $\Fl(d_1,\dots,d_r)$. The description is based on the one given in \cite[see Section 17]{MR680936} for the full flag case $\Fl(1,2,\dots,n-1)$.

Given a $m\times n$ generic matrix $X$, a maximal minor of $X$ is represented as a tuple of its column indices $[a_1,\dots,a_m]$ with $1\le a_1<\cdots<a_m\le n$. Now, fix a list of integers $1\le d_1<\cdots<d_r<n$, and let $X=(x_{kl})$ be a $d_r\times n$ generic matrix, and 
\[H=\{[a_1,\dots,a_{d_i}] \mid 1\le a_1<\dots<a_{d_i}\le n, \ \text{for}\ 1\le i\le r\}\]
be the set of all $d_i$-minors of $X$ for $1\le i\le r$. The set $H$ becomes a poset with respect to the partial order
\begin{equation}\label{equ: partial order flag poset}
    [a_1,\dots,a_{d_i}]\ge [b_1,\dots,b_{d_j}] \ \Leftrightarrow \ d_j\ge d_i,\ \text{and}\ a_1\ge b_1,\dots, a_{d_i}\ge b_{d_i}.
\end{equation}

The multi-homogeneous coordinate ring of the flag variety $\Fl(d_1,\dots,d_r)$ with respect to the embedding
\begin{equation}\label{equ: flag embedding}
    \Fl(d_1,\dots,d_r) \hookrightarrow \prod_{i=1}^r\gr(d_i,V) \hookrightarrow \prod_{i=1}^r\P\left(\bigwedge^{d_i} V\right),
\end{equation}
where each 
\[\gr(d_i,V) \hookrightarrow \P\left(\bigwedge^{d_i} V\right)=\P^{\binom{n}{d_i}-1}\]
is the \emph{Pl\"{u}cker embedding}, is the $\Bbbk$-subalgebra $R$ of the polynomial ring $\Bbbk[X]=\Bbbk[x_{kl} \mid 1\le k\le d_r,\ 1\le l\le n]$ generated by $H$. Note that since $R$ is assumed to have the standard grading, $\deg h=1$ for all $h\in H$.

Since $H$ is a subposet of the poset of all size minors of $X$, denoted as $\Delta(X)$, and $\Delta(X)$ is a distributive lattice (see \cite[before Lemma 4.10]{MR986492}), it follows that $H$ is a distributive lattice as well. 

\begin{remark}\label{rem: union poset tuples}
    Fix $d,n\in\N$ with $d\le n$, and define the set of $d$-tuples of integers
    \begin{equation}\label{equ: tuple poset}
        I(d,n)=\{(a_1,\dots,a_d) \mid 1\le a_1<\dots<a_d\le n\},
    \end{equation}
    which is a poset under the natural partial order
    \[(a_1,\dots,a_d)\ge (b_1,\dots,b_d) \ \Leftrightarrow \ a_1\ge b_1,\dots, a_d\ge b_d.\]
    
    Fix $1\le d_1<\cdots<d_r<n$ as before. The poset $H$ is in bijective correspondence with the set
    \begin{equation}\label{equ: union tuple sets}
        \bigcup_{i=1}^rI(d_i,n),
    \end{equation}
    where each $I(d_i,n)$ is defined similarly in (\ref{equ: tuple poset}). The set (\ref{equ: union tuple sets}) has the same partial order as $H$, described by (\ref{equ: partial order flag poset}).
    Hence, from now on we identify $H$ with (\ref{equ: union tuple sets}).
\end{remark}
 
\begin{proposition}\label{prop: classic flag is F-pure}
    The multi-homogeneous coordinate ring $R$ of $\Fl(d_1,\dots,d_r)$ is an $F$-pure Gorenstein graded ASL on $H$ over $\Bbbk$. 
\end{proposition}

\begin{proof}
    The ring $R$ is a graded ASL on $H$ over $\Bbbk$ by \cite[Proposition 7.6]{MR2900150}. We note that by \cite[Theorem 7.7]{MR2900150}, the cone $\spec(R)$ over $\Fl(d_1,\dots,d_r)$ is Gorenstein, so it follows that $R$ is Gorenstein. Assuming that the field $\Bbbk$ is $F$-finite of characteristic $p>0$, since $R$ is a Gorenstein ASL it follows from \cite[Corollary 5.3]{MR4626865} that $R$ is $F$-split, and hence $R$ is $F$-pure. 
\end{proof}

\subsection{The $F$-pure Threshold of Partial Flag Varieties}
We now begin the computation of the $F$-pure threshold of the multi-homogeneous coordinate ring $R$ of the partial flag variety $\Fl(d_1,\dots,d_r)$.  

Let us first consider the Grassmannian case. When $r=1$ and $d=d_1$, $X$ is a $d\times n$ generic matrix, and $H=I(d,n)$ becomes the set of maximal minors of $X$, denoted $\Gamma(X)$. The multi-homogeneous coordinate ring $R$ becomes the homogeneous coordinate ring of the Grassmannian $\gr(d,V)$ under the Pl\"{u}cker embedding, which is denoted $G(X)$. The ring $G(X)$ is a graded ASL on $\Gamma(X)$ over $\Bbbk$. Assuming that $\Bbbk$ is $F$-finite of characteristic $p$, $G(X)$ is an $F$-pure Gorenstein ring, hence by Theorem \ref{thm: fpt graded rings} the $F$-pure threshold of $G(X)$ equals $-a(G(X))$. Since $G(X)$ is an ASL on $\Gamma(X)$, it follows from \cite[Corollary 1.4]{MR1188581} via Theorem \ref{thm: ASL a-invariant} that $-a(G(X))=n$. We immediately get the following.

\begin{proposition}\label{prop: fpt Grassmannian}
One has $\fpt(G(X))=n$.
\end{proposition}

We seek to generalize Proposition \ref{prop: fpt Grassmannian} to the partial flag case. Since we have already noted that the multi-homogeneous coordinate ring $R$ of a partial flag variety is a Gorenstein graded ASL on $H$ it is enough to compute its $a$-invariant via the principal chain of $H$, which we describe next. 

\subsection{Construction of the Principal Chain in $H$}\label{subsec: principal chain construction}
Let $1\le d_1<\cdots<d_r \le n$ be fixed integers. The join of two elements $\xi,\nu\in H=\bigcup_{i=1}^rI(d_i,n)$, where $\xi=[a_1,\dots,a_{d_i}]\in I(d_i,n)$ and $\nu=[b_1,\dots,b_{d_j}]\in I(d_j,n)$
is
\[\xi\sqcup\nu = \big[\max(a_1,b_1),\dots,\max(a_{d_i},b_{d_i})\big] \enspace\text{if $d_j\ge d_i$.}\]
We describe an algorithm that constructs the principal chain $\mathcal{P}(H)=\xi_1,\xi_2,\dots,\xi_m$ of $H$. Note that the elements of $\mathcal{P}(H)$ can be divided into each subset $I(d_i,n)$ of $H$, hence $\mathcal{P}(H)=\mathcal{C}_1\cup\cdots\cup\mathcal{C}_r$, where each $\mathcal{C}_i=\xi_1^{(d_i)}<\dots<\xi_{m_i}^{(d_i)}$ is a chain in $I(d_i,n)$ and $\mathcal{C}_r<\cdots<\mathcal{C}_1$ in $H$. 

Each subchain $\mathcal{C}_i$ is constructed as follows. If $\xi_j^{(d_i)}=[a_1,\dots,a_{d_i}]\in I(d_i,n)$, whose last $d_i-d_{i-1}$ entries are not the consecutive integers $n-(d_i-d_{i-1})+1,\dots,n$, then define $\xi_{j+1}^{(d_i)}=[b_1,\dots,b_{d_i}]\in I(d_i,n)$ by
\begin{equation}\label{equ: entry of tuple}
    b_k 
= \begin{cases}
    a_k, & \text{if $a_{k+1}=a_k+1$ or $a_k=n$} \\
    a_k+1, & \text{otherwise}
\end{cases}
\end{equation}
(In other words, $\xi_{j+1}^{(d_i)}$ is obtained from $\xi_j^{(d_i)}$ by adding 1 to all entries of $\xi_j^{(d_i)}$ that precede a `gap').

\begin{itemize}
    \item The only minimal element of $H$ is $[1,2,\dots,d_r]\in I(d_r,n)$, hence $\xi_1=[1,2,\dots,d_r]$ is the first element of $\mathcal{P}(H)$.

    \item Suppose $\xi_1,\xi_2,\dots,\xi_k$ is defined, where $\xi_k=[a_1,\dots,a_{d_i}]\in \mathcal{C}_i$ for some $1\le i\le r$, and $\xi_k\ne[n-d_1+1,\dots,n]\in I(d_1,n)$. If the last $d_i-d_{i-1}$ entries of $\xi_k$ are not $n-(d_i-d_{i-1})+1,\dots,n-1,n$, then construct the next element $\xi_{k+1}$ in $\mathcal{C}_i$ via (\ref{equ: entry of tuple}). Otherwise, drop the entries $n-(d_i-d_{i-1})+1,\dots,n-1,n$ from $\xi_k$ and apply (\ref{equ: entry of tuple}) to its remaining entries of $\xi_k$ to create $\xi_{k+1}$, which is the first element of $\mathcal{C}_{i-1}$. 

    \item The process stops after a finite number of steps when one reaches 
    \[[n-d_1+1,\dots,n-1,\ n]\in I(d_1,n),\]
    which is the only maximal element of $H$. Thus, $\xi_m=[n-d_1+1,\dots,n]$ is always be the last element of $\mathcal{P}(H)$. 
\end{itemize}

\begin{example}\label{example: general principal chain}
Consider the flag variety $\Fl(2,3,5)$ of partial flags in a $7$-dimensional vector space. So, we have $n=7,\ d_3=5,\ d_2=3,\ d_1=2$. The underlying poset of the coordinate ring $R$ of $\Fl(2,3,5)$ is $H=I(2,7)\cup I(3,7)\cup I(5,7)$. The principal chain $\mathcal{P}(H)$ is

\begin{align*}
\xi_1 &= [1,\ 2,\ 3,\ 4,\ 5], \\
\xi_2 &= [1,\ 2,\ 3,\ 4,\ \bm{6}], \\
\xi_3 &= [1,\ 2,\ 3,\ \bm{5},\ \bm{7}], \\
\xi_4 &= [1,\ 2,\ \bm{4},\ \bm{6},\ 7],\\
\xi_5 &= [1,\ \bm{3},\ \bm{5}],\\
\xi_6 &= [\bm{2},\ \bm{4},\ \bm{6}],\\
\xi_7 &= [\bm{3},\ \bm{5},\ \bm{7}],\\
\xi_8 &= [\bm{4},\ \bm{6}],\\
\xi_9 &= [\bm{5},\ \bm{7}],\\
\xi_{10} &= [\bm{6},\ 7]
\end{align*}
(where boldface means the previous entry increased by 1).
\end{example}

We introduce a bit of notation, which is used in the proof of Theorem \ref{thm: a-inv flag variety}. One may express $\xi=[a_1,\dots,a_d]\in I(d,n)$ as $\xi=[\beta_0,\dots,\beta_s]$, where each $\beta_i$ are sets of consecutive entries called \emph{blocks}. Each block $\beta_i$ is followed by a \emph{gap} $\chi_i$, defined as a block of consecutive entries \emph{not} appearing in $\xi$. If $a_d=n$, then $\chi_s$ is empty. 

\begin{example}
    Let $\xi=[1,2,3,6,8,9,10]\in I(7,12)$. Then $\xi=[\beta_0,\beta_1,\beta_2]$, where $\beta_0=(1,2,3),\ \beta_1=(6),\ \beta_2=(8,9,10)$ are the blocks, and $\chi_0=(4,5),\ \chi_1=(7),\ \chi_2=(11,12)$ are the gaps. 
\end{example}

\begin{theorem}\label{thm: a-inv flag variety}
    Let $R$ be the multi-homogeneous coordinate ring of $\Fl(d_1,\dots,d_r)$ with respect to embedding (\ref{equ: flag embedding}). The $a$-invariant of $R$ satisfies
    \[-a(R)=n+d_r-d_1.\]
\end{theorem}

\begin{proof}
    We count the elements of $\mathcal{P}(H)$.
    Fix $2\le i\le d_r$. In the principal chain of $H$, suppose that the first $i$ elements $\xi_1,\dots,\xi_i$ are defined. It follows from the algorithm for constructing $\mathcal{P}(H)$ in \ref{subsec: principal chain construction} that $\xi_i=[b_1,\dots,b_{d_k}]$ for some $1\le k\le r$, where $b_j=j$ for $1\le j\le d_r-i+1$, and $b_j>j$ otherwise. 

    In particular, suppose the first $d_r-d_1$ elements $\xi_1,\dots,\xi_{d_r-d_1}$ of $\mathcal{P}(H)$ are defined, where $\xi_{d_r-d_1}=[b_1,\dots,b_{d_k}]$ with $b_j=j$ for $1\le j\le d_1+1$, and $b_j>j$ otherwise. Express $\xi_{d_r-d_1}=[\beta_0,\beta_1,\dots,\beta_s]$ in terms of its blocks of consecutive integers $\beta_i$. The first block is $\beta_0=(1,2,\dots,d_1)$, which may be identified with $[1,2,\dots,d_1]\in I(d_1,n)$. Each block $\beta_i$ of $\xi_{d_r-d_1}$ is followed by a gap $\chi_i$ of size 1 for $0\le i\le s-1$, hence the first entry of $\beta_1$ is $d_1+2$. The construction of the remaining elements of $\mathcal{P}(H)$ starting with $\xi_{d_r-d_1}$ is equivalent to constructing the principal chain of $I(d_1,n)$, which starts at $\beta_0$ and ends at $[n+d_1-1,\dots,n-1,n]$. But we know this chain consists of $n$ elements by \cite[Corollary 1.4]{MR1188581}. Therefore, adding these $n$ elements to the first $\xi_1,\dots,\xi_{d_r-d_1}$ gives a total of $|\mathcal{P}(H)|=n+d_r-d_1$. The conclusion follows from Remark \ref{rem: principal chain standard grading}.
\end{proof}

By Proposition \ref{prop: classic flag is F-pure} and Theorems \ref{thm: fpt graded rings} \& \ref{thm: a-inv flag variety}, we immediately get the following. 

\begin{corollary}\label{cor: fpt of flag}
Let $R$ be the multi-homogeneous coordinate ring of $\Fl(d_1,\dots,d_r)$ with respect to embedding (\ref{equ: flag embedding}) over an $F$-finite field of characteristic $p$. Then one has
\begin{equation}\label{equ: fpt type A flag varieties}
    \fpt(R) = n+d_r-d_1.
\end{equation}
\end{corollary}

\section{Generalized Flag Varieties}\label{sec: Generalized Flag Varieties}

A \emph{(generalized) flag variety} is defined to be the homogeneous space $G/Q$, where $G$ is a reductive group and $Q$ is a standard parabolic subgroup. When $Q=P$ is a maximal parabolic subgroup, $G/P$ is the \emph{Grassmannian variety}. In this section, we seek to generalize our results in section \ref{sec: classical flags} for a broader class of flag varieties. 

\subsection{Algebraic Groups}
We fix the following notations and assumptions throughout the rest of this paper, where we consult \cite{MR2015057} and \cite{MR2107324} as references.  

\begin{itemize}
    \item $G\ -$ a semisimple simply connected algebraic group over $\Bbbk$;
    \item $T\ -$ a maximal split torus in $G$;
    \item $B\ -$ a Borel subgroup of $G$ containing $T$;
    \item $W\ -$ the Weyl group of $G$ with respect to $T$;
    \item $\Phi\ -$ the corresponding root system associated to the pair $(G,T)$;
    \item $\Phi^+\ -$ the set of positive roots;
    \item $\Delta\ -$ the set of simple roots;
    \item $\Phi_I:=\Z I\cap\Phi\ -$ the set of roots of $\Phi$ generated by a subset $I\subseteq\Delta$;
    \item $Q:=Q_I\ -$ the parabolic subgroup of $G$ determined by a subset $I\subseteq\Delta$. Note that $Q_\emptyset=B$, and $Q$ is called \emph{standard} if $Q$ contains a fixed $B$;
    \item $P:=P_\alpha\ -$ the maximal parabolic subgroup of $G$ associated to $\alpha\in\Delta$ (in this case, $P_\alpha=Q_I$, where $I=\Delta\setminus\{\alpha\}$);
    \item $\G_m -$ the general linear algebraic group $\GL_1$ over $\Bbbk$ (the \emph{multiplicative group});
    \item $X(T):=\Hom(T,\G_m)\ -$ the character lattice of $T$;
    \item $X^\vee(T):=\Hom(\G_m,T)\ -$ the cocharacter lattice of $T$;
    \item $\alpha^\vee\ -$ the coroot of $\alpha\in\Phi$, belongs to $X^\vee(T)$;
    \item $\langle\cdot,\cdot\rangle\ -$ the perfect paring $X(T)\times X^\vee(T) \to \Z$;
    \item $\varpi_\alpha\ -$ the fundamental weight associated to $\alpha\in\Delta$ (is also associated to $P_\alpha$), defined by $\langle\varpi_\alpha,\beta^\vee\rangle=\delta_{\alpha,\beta}$ (Kronecker delta) for $\alpha,\beta\in\Delta$;
    \item $\mathcal{L}(\lambda)\ -$ the line bundle on the flag $G/Q$ associated to $\lambda\in X(T)$;
    \item $\rho_I := \frac{1}{2}\sum_{\alpha\in\Phi^+\setminus\Phi_I}\alpha\ -$ the character in $X(T)\otimes\Q$ associated to the parabolic subgroup $Q_I$;
    \item $\omega_{G/Q_I}:=\mathcal{L}(-2\rho_I)\ -$ the canonical bundle on the flag $G/Q_I$;
    \item $X(Q_I):=\{\lambda\in X(T)\mid \langle\lambda,\alpha^\vee\rangle=0,\ \forall\alpha\in I\}\ -$ the character lattice of a parabolic subgroup $Q_I$ \cite[(4) on page 169]{MR2015057}
\end{itemize}

\begin{example}\label{ex: classic flag}
    In the special linear group $\SL_n$, which is a semisimple simply connected algebraic group, $T$ and $B$ are respectively the group of dialgonal and upper triangular matrices. The root system $\Phi$ associated to the pair $(\SL_n,T)$ is of type $\mathbf{A}_{n-1}$, with $\Delta=\{\alpha_1,\dots,\alpha_{n-1}\}$ as its set of simple roots. For an $n$-dimensional vector space $V$, we had defined in Section \ref{sec: classical flags} the partial flag variety $\Fl(d_1,\dots,d_r)$ of all partial flags of $V$ with signature $(d_1,\dots,d_r)$. The variety $\Fl(d_1,\dots,d_r)$ can be identified with the homogeneous space $\SL_n/Q$ where $Q$ is a standard parabolic subgroup determined by the set of simple roots $I=\Delta\setminus\{\alpha_{d_1},\dots,\alpha_{d_r}\}$ with $1\le d_1<\cdots<d_r\le n-1$. The parabolic subgroup $Q$ is the group of all block upper triangular matrices of the form
    \[\begin{pmatrix}
    A_1 & \ast & \cdots & \ast & \ast \\
    \vdots & \vdots & \ddots & \vdots & \vdots \\
    0 & 0 & \cdots & A_r & \ast \\
    0 & 0 & \cdots & 0 & A
    \end{pmatrix},\]
    where $A_i$ is a square matrix of size $d_i-d_{i-1}$ for $1\le i\le r$, and $A$ is a square matrix of size $n-d_r$. In the case when $I=\Delta\setminus\{\alpha_d\}$, $Q$ becomes the maximal parabolic subgroup $P_d$ (with $1\le d\le n-1$) of all block upper triangular matrices of the form
    \[\begin{pmatrix}
        \ast & \ast \\
        O & \ast
    \end{pmatrix},\]
    where $O$ is the zero matrix of size $(n-d)\times d$. Thus, the Grassmannian variety $\gr(d,V)$ we defined in Section \ref{sec: classical flags} can be identified with $\SL_n/P_d$.
\end{example}

\subsection{Embeddings Induced by Dominant Weights}\label{sec: Flag varieties under other embeddings}

The condition that $G$ is semisimple implies  
\[X(T)\cong\bigoplus_{\alpha\in\Delta}\Z\varpi_\alpha,\]
hence every weight $\lambda\in X(T)$ may be written as
$\lambda=\sum_{\alpha\in \Delta}n_{\alpha}\varpi_{\alpha}$, where 
$n_{\alpha}=\langle\lambda,\alpha^\vee\rangle$ follows from the definition of the fundamental weight $\varpi_{\alpha}$. 

Fix a parabolic subgroup $Q=Q_I$. Then one also has 
\[X(Q_I)\cong\bigoplus_{\alpha\in\Delta\setminus I}\Z\varpi_\alpha.\]
A weight $\lambda\in X(Q_I)$ is called \emph{$Q$-dominant} if $\langle\lambda,\alpha^\vee\rangle\ge0$ for all $\alpha\in\Delta\setminus I$, and $\lambda$ is called \emph{regular $Q$-dominant} if $\ge$ is replaced by the strict inequality $>$. 

The group homomorphism $X(Q_I)\to\pic(G/Q_I)$, defined by $\lambda\mapsto\mathcal{L}(\lambda)$, is an isomorphism. Hence every line bundle on $G/Q_I$ is of the form $\mathcal{L}(\lambda)$ for some $\lambda\in X(Q_I)$. Note that if $Q_I$ is maximal parabolic subgroup, then $\pic(G/Q_I)\cong\Z$. One has the important property that $\lambda\in X(Q_I)$ is $Q$-dominant if and only if $H^0(G/Q,\mathcal{L}(\lambda))\ne0$ \cite[Proposition on page 178]{MR2015057}. Also, for any $\lambda\in X(Q_I)$ one has that [$\mathcal{L}(\lambda)$ is very ample] $\Leftrightarrow$ [$\mathcal{L}(\lambda)$ is ample] $\Leftrightarrow$ [$\langle \lambda,\alpha^{\vee}\rangle>0$ for all $\alpha\in \Delta\setminus I$, i.e., $\lambda$ is regular $Q$-dominant] (for the first equivalence, see \cite[Exercise 3.1.E.(1)]{MR2107324}, and for the second, see \cite[Remarks (1), page 204]{MR2015057}). 

Let $G$ be a semisimple simply connected algebraic group over $\Bbbk$, fix a maximal torus and Borel subgroup $T\subseteq B$, and let $Q\supseteq B$ be a parabolic subgroup.
Recall that a (very) ample line bundle on a flag variety $G/Q$ takes the form $\mathcal{L}(\lambda)$, for some regular $Q$-dominant weight $\lambda$. This induces the natural embedding 
\begin{equation}\label{equ: standard embedding}
    G/Q \hookrightarrow\P(H^0(G/Q,\mathcal{L}(\lambda))^*),
\end{equation}
in which $G/Q$ is projectively normal under \cite[Theorem 1]{MR778124}, hence the homogeneous coordinate ring of $G/Q$ under this embedding is the section ring
\begin{equation}\label{equ: section ring, dom.wt.}
    R_{\lambda} := \bigoplus_{n\ge0}H^0(G/Q,\mathcal{L}(n\lambda)).
\end{equation}
By projective normality of $G/Q$, $R_\lambda$ is normal and generated in degree 1. In addition, it is Cohen-Macaulay by \cite[Theorem 5]{MR788411}. 
We make the verification that $R_{\lambda}$ is $F$-pure, which is immediate from results in the literature. 

\subsection{The $F$-pure Threshold of the Section Rings $R_\lambda$}

One says a projective variety $X$ over a field $\Bbbk$ of characteristic $p$ is \emph{$F$-split} if the Frobenius map $\mathcal{O}_X\to F_*\mathcal{O}_X, \ f\mapsto f^p$ splits as a map of $\mathcal{O}_X$-modules.

\begin{lemma}\label{lem: flags are F-pure}
For every regular $Q$-dominant weight $\lambda$, the section ring $R_\lambda$ of a flag variety $G/Q$ with respect to $\mathcal{L}(\lambda)$ over a field $\Bbbk$ of characteristic $p$ is $F$-pure.
\end{lemma}

\begin{proof}
By \cite[Theorem 2]{MR799251}, for any standard parabolic subgroup $Q$\textsuperscript{1}\footnote{\textsuperscript{1}One requires the parabolic subgroup $Q$ to be a reduced subscheme, otherwise the flag variety $G/Q$ is not $F$-split (see \cite{MR1250534}). All parabolic subgroups in our examples are reduced, so we always make this assumption about $Q$.} of $G$, the flag variety $G/Q$ over an algebraically closed field is $F$-split. Since the line bundles $\mathcal{L}(\lambda)$ for regular $Q$-dominant weights $\lambda$ are very ample, it follows from \cite[Proposition 3.1]{MR1786505} that the section rings $R_\lambda$ of $G/Q$ are $F$-pure over an algebraically closed field. It follows that if $R_\lambda$ is defined over an arbitrary field $\Bbbk$ of positive characteristic, by taking its algebraic closure $\overline{\Bbbk}$, we see that $R_\lambda\otimes_\Bbbk\overline{\Bbbk}$ is $F$-pure, which implies that $R_\lambda$ is $F$-pure (see \cite[Corollary 7.1.22]{MR4627943}).  
\end{proof}

There is a useful characterization of the Gorensteinness of Veronese rings that we need, stated as follows:

\begin{theorem}\cite[Corollary 3.1.5, Theorem 3.2.1]{MR494707}
\label{thm: Veronese Gorenstein}
Suppose that $R$ is a Cohen-Macaulay standard graded $\Bbbk$-algebra of $\dim R\ge2$. The $n$\textsuperscript{th} Veronese $R^{(n)}$ is Gorenstein if and only if $R$ is Gorenstein and $a(R)\equiv 0 \Mod{n}$.
\end{theorem}

We shall use this fact to compute the $F$-pure thresholds of some $R_\lambda$, which also includes one that is \emph{not} a Gorenstein ring whose value is strictly rational.

\begin{proposition}\label{prop: non-gorenstein flag}
Suppose the underlying field is $F$-finite of characteristic $p$. Let $I\subseteq\Delta$ be a subset, $Q=Q_I$ be the corresponding parabolic subgroup, and $\rho_I$ be the associated dominant weight. If $I\ne\Delta\setminus\{\alpha\}$ for any $\alpha\in\Delta$ (so $Q$ is \emph{not} a maximal parabolic subgroup), one has the following:
\begin{enumerate}
    \item [(1)] The section ring (\ref{equ: section ring, dom.wt.}) for $\lambda=\rho_I$, that is $R_{\rho_I}$, of $G/Q$ (not a Grassmannian) is Gorenstein.

    \item [(2)] Let $\lambda=n\rho_I$, for any $n\in\N$. Then $R_\lambda$ is $F$-pure. 

    \item [(3)] $\fpt(R_{\rho_I})=2$.

    \item [(4)] If $\lambda=n\rho_I$, for some $n\in\N$, then $\fpt(R_\lambda)=2/n$. In particular, $R_\lambda$ is Gorenstein if and only if $n=1$ or 2.
\end{enumerate}
\end{proposition}

\begin{proof}
\begin{enumerate}
    \item [(1)] The section ring of the anticanonical bundle $\omega_{G/Q}^{-1}=\mathcal{L}(2\rho_I)$ is
    \[S = \bigoplus_{n\ge0}H^0(G/Q,\omega_{G/Q}^{-n}) = \bigoplus_{n\ge0}H^0(G/Q,\mathcal{L}(2n\rho_I)) = R_{\rho_I}^{(2)},\]
    which is the 2\textsuperscript{nd} Veronese of $R_{\rho_I}$. Note that $S=R_{\rho_I}^{(2)}$ is a direct summand of $R_{\rho_I}$ as an $S$-module and that $R_{\rho_I}$ is integral over $S$. It follows that $R_{\rho_I}$ being a Cohen-Macaulay ring implies that $S$ is a Cohen-Macaulay ring. The graded canonical module of $S$ is $\omega_S\cong S(-1)$, by \cite[(5.1.8)]{MR494707}, hence $S$ is Gorenstein by \cite[Proposition 2.1.3]{MR494707}. It follows from Theorem \ref{thm: Veronese Gorenstein} that $R_{\rho_I}$ is Gorenstein, and moreover $a(R_{\rho_I})$ is divisible by 2.

    \item [(2)] In the proof of part (1) we showed $\omega_S\cong S(-1)$, hence by \cite[6.2. Proposition (i)]{MR1786505}, $G/Q$ is a Fano variety, i.e., its anticanonical bundle $\omega_{G/Q}^{-1}=\mathcal{L}(2\rho_I)$ is ample. Then $\mathcal{L}(2\rho_I)$ being ample implies $2\rho_I$ is regular $Q$-dominant, hence $\rho_I$ is as well. It is easily deduced that $\lambda=n\rho_I$ is regular $Q$-dominant for any $n$. Thus, Lemma \ref{lem: flags are F-pure} implies that $R_\lambda$ is $F$-pure. 

    \item [(3)] One concludes 
    \[-1=a(S) = a\left(R_{\rho_I}^{(2)}\right) = \left\lceil\frac{a(R_{\rho_I})}{2}\right\rceil = \frac{a(R_{\rho_I})}{2},\]
    hence it follows $\fpt(R_{\rho_I})=-a(R_{\rho_I})=2$. 

    \item [(4)] One sees that $R_\lambda=R_{\rho_I}^{(n)}$ is the $n$\textsuperscript{th} Veronese of $R_{\rho_I}$, so the $F$-pure threshold of its maximal ideal is $\fpt(R_\lambda)=\frac{1}{n}\fpt(R_{\rho_I})=\frac{2}{n}$. The Gorenstein statement follows from Theorem \ref{thm: Veronese Gorenstein}. 
\end{enumerate}
\end{proof}

\subsection{General Grassmannians}

Let $\Delta=\{\alpha_1,\dots,\alpha_r\}$ be an enumeration of the simple roots. For some $1\le d\le r$, let $P=P_d$ be the maximal parabolic subgroup of $G$ determined by $I=\Delta\setminus\{\alpha_d\}$, let $\varpi_d:=\varpi_{\alpha_d}$ be the corresponding fundamental weight, and $\mathcal{L}_d:=\mathcal{L}(\varpi_d)$, which is the ample generator of $\pic(G/P)=\Z$. When $\lambda=\varpi_d$, then (\ref{equ: standard embedding}) is the Pl\"{u}cker embedding of $G/P$ into projective space, and (\ref{equ: section ring, dom.wt.}) is the corresponding section ring 
\begin{equation}\label{equ: grass. sec. ring}
R_{\varpi_d}:=\bigoplus_{n\ge0}H^0(G/P,\mathcal{L}_d^n).
\end{equation}

\begin{proposition}\label{prop: grass. sec. ring F-pure}
    Over an $F$-finite field of characteristic $p$,
    the ring $R_{\varpi_d}$ is $F$-pure.
\end{proposition}

\begin{proof}
    The fundamental weight $\varpi_d$ is regular $P$-dominant, since $\Delta\setminus I=\{\alpha_d\}$, and hence $\langle\varpi_d,\alpha_d^\vee\rangle=\delta_{\alpha_d,\alpha_d}=1$. It follows from Lemma \ref{lem: flags are F-pure} that $R_{\varpi_d}$ is $F$-pure.
\end{proof}

We prove $R_{\varpi_d}$ is Gorenstein. In addition, it follows from the proof that the value of the $a$-invariant of $R_{\varpi_d}$ can be determined by summing a certain subset of positive roots of the root system associated to $G$. With this method one can compute the $F$-pure thresholds of the coordinate rings of some Grassmannians that are not an ASL (see subsection \ref{subsec: Computing with Root Systems} and Table \ref{table: exceptional}).

\begin{proposition}\label{prop: Grassmannian Gorenstein a-invariant}
The homogeneous coordinate ring $R_{\varpi_d}$ of the Grassmannian $G/P$ under the Pl\"{u}cker embedding is Gorenstein.
\end{proposition}

\begin{proof}
Let $\omega_{G/P}^{-1}$ be the anticanonical bundle on $G/P$. Then $\omega_{G/P}^{-1}\in\pic(G/P)=\Z $ implies $\omega_{G/P}^{-1}=\mathcal{L}_d^m$, for some $m\in\Z$ (we may assume $m>0$). If $S$ is the corresponding section ring of $\omega_{G/P}^{-1}$, then $S=R_{\varpi_d}^{(m)}$ is the $m$\textsuperscript{th} Veronese of $R_{\varpi_d}$. Note that $S$ is a direct summand of $R_{\varpi_d}$ as an $S$-module and that $R_{\varpi_d}$ is integral over $S$. It follows that $R_{\varpi_d}$ being a Cohen-Macaulay ring implies that $S$ is a Cohen-Macaulay ring. By \cite[(5.1.8)]{MR494707}, the graded canonical module of $S$ is $\omega_S = \bigoplus_{n\in\Z}H^0(G/P,\omega_{G/P}^{-n-1}) \cong S(-1)$ as graded rings, hence $S=R_{\varpi_d}^{(m)}$ is Gorenstein by \cite[Proposition 2.1.3]{MR494707}. Therefore, $R_{\varpi_d}$ is Gorenstein by Theorem \ref{thm: Veronese Gorenstein}. 
\end{proof}

\begin{corollary}\label{cor: a-invariant via root systems}
Let $\Phi$ be the root system associated to $G$. Continuing the assumptions of this subsection, the $a$-invariant $a(R_{\varpi_d})$ of the coordinate ring $R_{\varpi_d}$ of $G/P$ satisfies 
\begin{equation}\label{equ: a-invariant root system}
    -a(R_{\varpi_d})\varpi_d = \sum_{\alpha\in\Phi^+\setminus\Phi_I}\alpha.
\end{equation}
\end{corollary}

\begin{proof}
In the conclusion of the proof of Proposition \ref{prop: Grassmannian Gorenstein a-invariant}, Theorem \ref{thm: Veronese Gorenstein} also says that $a(R_{\varpi_d})\equiv 0 \Mod{m}$. Since $S=R_{\varpi_d}^{(m)}$ is Gorenstein, $\omega_S\cong S(-1)$ implies $a(S)=-1$. One has
\[-1 = a(S) = a\left(R_{\varpi_d}^{(m)}\right) = \left\lceil \frac{a(R_{\varpi_d})}{m} \right\rceil = \frac{a(R_{\varpi_d})}{m},\]
and hence $m=-a(R_{\varpi_d})$. It then follows 
\[\mathcal{L}(2\rho_I) = \omega_{G/P}^{-1} = \mathcal{L}_d^m = \mathcal{L}(m\varpi_d) = \mathcal{L}(-a(R_{\varpi_d})\varpi_d),\]
which implies $-a(R_{\varpi_d})\varpi=2\rho_I$, since the map $\lambda\mapsto\mathcal{L}(\lambda)$ is one-to-one. Equation (\ref{equ: a-invariant root system}) immediately follows by recalling that $\rho_I:= \frac{1}{2}\sum_{\alpha\in\Phi^+\setminus\Phi_I}\alpha$.
\end{proof}

We are also able to compute the $F$-pure threshold of non-Gorenstein coordinate rings of the Grassmannian $G/P$, which are strictly rational numbers. 

\begin{proposition}\label{prop: non-gorenstein grassmannians}
Suppose the underlying field is $F$-finite of characteristic $p$. For a regular $P$-dominant weight $\lambda$, one has
\begin{equation}
    \fpt(R_\lambda) = \frac{-a(R_{\varpi_d})}{\langle\lambda,\alpha_d^\vee\rangle}.
\end{equation}
Moreover, $R_\lambda$ is Gorenstein if and only if $a(R_{\varpi_d})\equiv 0 \Mod{\langle\lambda,\alpha_d^\vee\rangle}$.
\end{proposition}

\begin{proof}
Every regular $P$-dominant weight is of the form $\lambda=n_d\varpi_d$, where $n_d=\langle \lambda,\alpha_d^\vee\rangle>0$, hence $\mathcal{L}(\lambda)=\mathcal{L}_d^{n_d}$. Therefore, $R_\lambda = R_{\varpi_d}^{(n_d)}$ is a Veronese subring of $R_{\varpi_d}$. The $F$-pure threshold of a Veronese ring is $\fpt(R^{(n)})=\frac{1}{n}\fpt(R)$ \cite[Proposition 2.2.(2)]{MR2097584}.
By Lemma \ref{lem: flags are F-pure} and Proposition \ref{prop: Grassmannian Gorenstein a-invariant} we know that $R_{\varpi_d}$ is an $F$-pure Gorenstein ring, so $\fpt(R_{\varpi_d})=-a(R_{\varpi_d})$, and it follows that
\[\fpt(R_\lambda) = \fpt\left(R_{\varpi_d}^{(n_d)}\right) = \frac{\fpt(R_{\varpi_d})}{n_d}=\frac{-a(R_{\varpi_d})}{\langle\lambda,\alpha_d^\vee\rangle} \in\Q.\]
If $R_\lambda=R_{\varpi_d}^{(n_d)}$ is Gorenstein, then Theorem \ref{thm: Veronese Gorenstein} implies $a(R_{\varpi_d})\equiv 0 \Mod{n_d}$ (we already know that $R_{\varpi_d}$ is Gorenstein).
On the other hand, if $a(R_{\varpi_d})\equiv 0 \Mod{n_d}$, then since we know $R_{\varpi_d}$ is Gorenstein, Theorem \ref{thm: Veronese Gorenstein} implies that $R_\lambda=R_{\varpi_d}^{(n_d)}$ is Gorenstein, hence $\fpt(R_\lambda)=-a(R_\lambda)\in\N$.
\end{proof}

\subsection{Minuscule Grassmannians}\label{subsection: Minuscule Grassmannians}

We have seen in Section \ref{sec: classical flags} that the coordinate rings of classical flag varieties have the structure of an ASL. Other types of flag varieties whose coordinate rings are known to have this property are minuscule Grassmannians. Using the ASL structure we compute the $F$-pure threshold of these rings, listed in Table \ref{table: minuscule}, via the principal chain. We consult \cite{MR2416742} and \cite{MR866020} for the background on this section.

\begin{definition}
A fundamental weight $\varpi$ is called \emph{minuscule} if $\langle\varpi,\alpha^{\vee}\rangle\le1$ for all $\alpha\in\Phi^+$. The maximal parabolic subgroup $P$ associated to $\varpi$ is called a \emph{minuscule parabolic subgroup}, and the corresponding homogeneous space $G/P$ is called a \emph{minuscule Grassmannian}. 
\end{definition}

\begin{example}
First, given a general root system $\Phi$ in a Euclidean vector space $E$ with inner product $(\cdot,\cdot)$, for any two roots $\alpha,\beta\in\Phi$ define 
\[\langle\beta,\alpha^\vee\rangle := \frac{2(\beta,\alpha)}{(\alpha,\alpha)}.\]
When $\Phi$ is a root system of type $\mathbf{B}_2$, its set of positive roots is \[\Phi^+=\{\alpha_1,\alpha_2,\alpha_3,\alpha_4\}=\{e_1,\ e_2,\ e_1-e_2,\ e_1+e_2\},\]
where $e_1=(1,0)$ and $e_2=(0,1)$ are the coordinate vectors of $\R^2$. Consider the fundamental weight $\varpi_2=\frac{1}{2}(e_1+e_2)$ (the general description of these is listed in subsection \ref{subsec: Computing with Root Systems}). We see that
\[\langle\varpi_2,e_1^\vee\rangle = \frac{2\left(\frac{1}{2}(e_1+e_2),\ e_1\right)}{(e_1,\ e_1)} = 1.\]
Similarly, we get $\langle\varpi_2,e_2^\vee\rangle=1,\ \langle\varpi_2,(e_1-e_2)^\vee\rangle=\langle\varpi_2,(e_1+e_2)^\vee\rangle=0$, and hence $\varpi_2$ is a minuscule weight.  
\end{example}

One has a complete list of the minuscule weights for the irreducible root systems of the following types:
\begin{align*}
    &\text{Type $\mathbf{A}_{n-1}$} : \text{All the fundamental weights $\varpi_1,\dots,\varpi_{n-1}$}\\
    &\text{Type $\mathbf{B}_n$} : \varpi_n \\
    &\text{Type $\mathbf{C}_n$} : \varpi_1 \\
    &\text{Type $\mathbf{D}_n$} : \varpi_1,\varpi_{n-1},\varpi_n \\
    &\text{Type $\mathbf{E}_6$} : \varpi_1,\varpi_6 \\
    &\text{Type $\mathbf{E}_7$} : \varpi_7 
\end{align*}
(for a description of the these weights, see Section \ref{subsec: Computing with Root Systems}). The root systems of types $\mathbf{E}_8,\mathbf{F}_4,\mathbf{G}_2$ do not have minuscule weights.

\begin{definition}
Given a maximal parabolic subgroup $P_d$ of $G$ with associated minuscule weight $\varpi_d$, one calls 
\[\mathbf{X}_n(\varpi_d):=W/W_{P_d}\]
the \emph{minuscule lattice}, where $W_{P_d}$ is the Wyel group of $P_d$, and $\mathbf{X}_n$ denotes the type ($\mathbf{A}_\bullet, \mathbf{B}_\bullet, \mathbf{C}_\bullet, \mathbf{D}_\bullet, \mathbf{E}_\bullet$) of the root system $\Phi$.
\end{definition}

It is known that $\mathbf{X}_n(\varpi_d)$ is a distributive lattice (see \cite[Remark 5.3]{MR2416742}). A few minuscule lattices important for our purposes are described below.

\subsection*{Type $\mathbf{A}_{n-1}$}
For each $1\le d\le n-1$, the minuscule lattice associated to $\varpi_d$ is
\[\mathbf{A}_{n-1}(\varpi_d)=W_{\SL_n}/W_{P_d} = \frac{S_n}{S_d\times S_{n-d}}\cong I(d,n),\]
where $S_k$ for $k\in\N$ is the symmetric group on $k$ letters.

\subsection*{Type $\mathbf{B}_{n-1}$} The minuscule lattice associated to $\varpi_{n-1}$ is
\[\mathbf{B}_{n-1}(\varpi_{n-1}) = \left\{(i_1,\dots,i_{n-1})\in I(n-1,2n) \ \Bigg| \
\begin{aligned}
&\text{for}\ 1\le j\le n-1,\\
&\{j,\ 2n-j\}\cap\{i_1,\dots,i_{n-1}\}\ \text{is size 1}
\end{aligned}
\right\}.\]

\subsection*{Type $\mathbf{D}_{n}$}
The minuscule lattice associated to $\varpi_n$ is 
\[\mathbf{D}_{n}(\varpi_{n}) = \left\{(i_1,\dots,i_{n})\in I(n,2n) \ \Bigg| \
\begin{aligned}
&\text{for}\ 1\le j\le n, \\
&\{j,\ 2n+1-j\}\cap\{i_1,\dots,i_{n}\} \ \text{is size 1}, \\
&\mid\{j,\ 1\le j\le n\mid i_j>n\}\mid\ \text{is even}
\end{aligned} \right\}.\]
One has the lattice isomorphisms
\[\mathbf{B}_{n-1}(\varpi_{n-1}) \cong \mathbf{D}_{n}(\varpi_{n-1}) \cong \mathbf{D}_{n}(\varpi_{n}).\]

\begin{proposition}\label{prop: principal chain type B}
The cardinality of the principal chain in the minuscule lattice $\mathbf{B}_{n-1}(\varpi_{n-1})\cong \mathbf{D}_{n}(\varpi_{n})$ is $2(n-1)$, for $n\ge2$. 
\end{proposition}

\begin{proof}
Consider the lattice $\bigcup_{d=0}^{n-1}I(d,n-1)$, where $I(0,n-1):=\{\hat{1}\}$ consists of its maximal element, and $I(n-1,n-1):=\{(1,2,\dots,n-1)\}$ consists of the $(n-1)$-length tuple, which is its minimal element $\hat{0}$. In particular, this lattice is just $H_{n-1}\cup\{\hat{0},\hat{1}\}$, where $H_{n-1}:=\bigcup_{d=1}^{n-2}I(d,n-1)$ is the type of lattice in (\ref{equ: union tuple sets}). For $w=(w_1,\dots,w_{n-1})\in\mathbf{B}_{n-1}(\varpi_{n-1})$, let $1\le d\le n-1$ be such that $w_d\le n-1$ and $w_{d+1}\ge n$. Note that $w$ is determined by $(w_1,\dots,w_d)$, so we define the map
\begin{align*}
    \theta: \mathbf{B}_{n-1}(\varpi_{n-1}) &\longrightarrow \bigcup_{d=0}^{n-1}I(d,n-1) \\
    (w_1,\dots,w_d,\dots,w_{n-1}) &\longmapsto (w_1,\dots,w_d) \\
    (1,2,\dots,n-2,n-1) &\longmapsto \hat{0} \enspace(\text{maps to itself})\\
    (1,2,\dots,n-2,n+1) &\longmapsto \hat{1}.
\end{align*}
Order is preserved in each $I(d,n)$, for $1\le d\le n-2$, hence $\theta$ is an order homomorphism. It is also clearly injective. The lattice $\mathbf{B}_{n-1}(\varpi_{n-1})$ is of size $2^{n-1}$ (which is clear by the condition on its elements), while the size of $\bigcup_{d=0}^{n-1}I(d,n-1)$ is
\[\sum_{d=0}^{n-1}| I(d,n-1)|  = \sum_{d=0}^{n-1}\binom{n-1}{d} = 2^{n-1}.\]
It follows that $\theta$ is surjective as well, hence it is an isomorphism of lattices. The principal chain in $H_{n-1}\cup\{\hat{0},\hat{1}\}$ is just the principal chain of $H_{n-1}$ with $\hat{0}$ and $\hat{1}$ added, so its cardinality is 
\[|\mathcal{P}(H_{n-1})\cup\{\hat{0},\hat{1}\})| = |\mathcal{P}(H_{n-1})|+2 = 2((n-1)-1)+2 = 2(n-1).\]
The principal chain of $\mathbf{B}_{n-1}(\varpi_{n-1})$ maps onto $\mathcal{P}(H_{n-1})\cup\{\hat{0},\hat{1}\}$ by $\theta$, hence it has the same size.
\end{proof}

The work of C.S. Seshadri in \cite{MR541023} implies that the homogeneous coordinate ring $R_{\varpi_d}$ of a minuscule Grassmannian $G/P_d$ under the Pl\"{u}cker embedding is an ASL on $\mathbf{X}_n(\varpi_d)$. The $F$-purity and Gorensteinness of $R_{\varpi_d}$ follows from Propositions \ref{prop: grass. sec. ring F-pure} and \ref{prop: Grassmannian Gorenstein a-invariant} respectively. We present the main result of this subsection.

\begin{theorem}
Let $R_{\varpi_d}$ be the homogeneous coordinate ring of the Grassmannian $G/P_d$ under the Pl\"{u}cker embedding over an $F$-finite field of characteristic $p$. We have the following table:

\begin{table}[H]
\centering 
\captionsetup{justification=centering}
\caption{The $F$-pure threshold of minuscule Grassmannians}
\begin{tabular}{ |p{1cm}|p{3cm}|p{3cm}|p{2cm}|  }
\hline
Type & Weight index $d$ & Minuscule Grassmannians & $\fpt(R_{\varpi_d})$ \\
\hline
$\mathbf{A}_{n-1}$ & $1,\dots,n-1$ & $\SL_n/P_d$ & $n$ \\
$\mathbf{B}_n$ & $n$ & $\SO_{2n+1}/P_d$  & $2n$ \\
$\mathbf{C}_{n}$ & $1$ & $\Sp_{2n}/P_d$ & $2n$ \\
$\mathbf{D}_{n}$ & $1,n-1,n$ & $\SO_{2n}/P_d$ & $2(n-1)$ \\
$\mathbf{E}_6$ & $1,6$ & $E_6/P_d$ & $12$ \\
$\mathbf{E}_7$ & $7$ & $E_7/P_d$ & $18$ \\
\hline
\end{tabular}
\label{table: minuscule}
\end{table}
\end{theorem}

\begin{proof}
We compute all the values of Table \ref{table: minuscule} using the formula 
\[\fpt(R_{\varpi_d})=-a(R_{\varpi_d})=\text{size of principal chain in $\mathbf{X}_n(\varpi_d)$}.\]
Note that for types $\mathbf{C}_n$ and $\mathbf{D}_n$ of weight index $d=1$, $R_{\varpi_d}$ is simply enough that the $a$-invariant can be computed directly.  
\begin{itemize}
    \item Type $\mathbf{A}_{n-1}$: Since $\SL_n/P_d\cong\gr(d,\Bbbk^n)$ (see Example \ref{ex: classic flag}), then $R_{\varpi_d}=G(X)$, and this is Proposition \ref{prop: fpt Grassmannian}.
    
    \item Type $\mathbf{B}_n$: This follows from Proposition \ref{prop: principal chain type B}.
    
     \item Type $\mathbf{C}_n$: Since $\Sp_{2n}/P_1$ is the projective space $\P^{2n-1}$, its coordinate ring is the polynomial ring $R_{\varpi_1}=\Bbbk[x_0,\dots,x_{2n-1}]$. Hence $\fpt(R_{\varpi_1})=\dim R_{\varpi_1}=2n$.
     
     \item Type $\mathbf{D}_n$: For the weight indices $d=n-1,n$ recall the lattice isomorphisms $\mathbf{D}_{n}(\varpi_{n-1}) \cong \mathbf{D}_{n}(\varpi_{n})\cong\mathbf{B}_{n-1}(\varpi_{n-1})$. Proposition \ref{prop: principal chain type B} then gives the value of the $F$-pure threshold of $R_{\varpi_{n-1}}$ and $R_{\varpi_n}$. 
     
     For weight index $d=1$, the Grassmannian $\SO_{2n}/P_1$ is defined by the quadratic $f=\sum_{i=1}^nx_iy_{n+1-i}$, hence its coordinate ring is the hypersurface ring
     \[ R_{\varpi_1} =\frac{S}{(f)} =\frac{\Bbbk[x_1,\dots,x_n,y_1,\dots,y_n]}{\left(\sum_{i=1}^nx_iy_{n+1-i}\right)}.\]
     Its $a$-invariant is $a(R_{\varpi_1})=a(S)+\deg f = -2n+2$, so $\fpt(R_{\varpi_1})=2(n-1)$.
     
     \item Types $\mathbf{E}_6$ and $\mathbf{E}_7$:
     The Hasse diagrams of $\mathbf{E}_6(\varpi_6)\cong \mathbf{E}_6(\varpi_1)$ and $\mathbf{E}_7(\varpi_7)$ can be found in \cite[figure 1]{MR782055}. From the diagrams one can conclude that the size of the
     principal chains of $\mathbf{E}_{6}(\varpi _{6})$ and $\mathbf{E}_{7}(\varpi _{7})$ are 12 and 18 respectively.
\end{itemize}
\end{proof}

\subsection{Computing With Root Systems}\label{subsec: Computing with Root Systems}

We use formula (\ref{equ: a-invariant root system}) of Corollary \ref{cor: a-invariant via root systems} to compute the $a$-invariant, and hence the $F$-pure threshold, of $R_{\varpi_d}$ for Grassmannians $G/P_d$ when $\Phi=\Phi(G,T)$ is any of the nine irreducible root systems. For convenience, we list below the descriptions of $\Phi^+$, $\Delta$, and the fundamental weights for each of these systems, following \cite{MR1153249}. 

\begin{remark}
Each of the root systems below is viewed inside of a Euclidean vector space $E$ with an inner product, while the root systems $\Phi$ we consider in this paper are associated to a semisimple group $G$. While these two concepts are not the same as sets, they are the same up to isomorphism, so we identify $\Phi$ with the root systems below.
\end{remark}

\subsection*{Type $\mathbf{A}_{n-1}$} The positive roots are $\Phi^+=\{e_i-e_j \mid 1\le i<j \le n\}$, the simple roots are $\Delta=\{e_i-e_{i+1}\mid 1\le i\le n-1\}$, and the fundamental weights are $\varpi_i=e_1+\cdots+e_i$, for $1\le i\le n-1$. 

\subsection*{Type $\mathbf{B}_n$} The positive roots are $\Phi^+=\{e_i,\ e_i\pm e_j \mid 1\le i<j \le n\}$, the simple roots are $\Delta=\{e_i-e_{i+1},\ e_n\mid 1\le i\le n-1\}$, the fundamental weights are
\[\begin{cases}
   &\varpi_n = \frac{1}{2}(e_1+\cdots+e_n)\\
   &\varpi_i = e_1+e_2+\cdots+e_i, \enspace\text{for $1\le i\le n-1$}.
\end{cases}\]

\subsection*{Type $\mathbf{C}_n$} The positive roots are $\Phi^+=\{2e_i,\ e_i\pm e_j \mid 1\le i<j \le n\}$, the simple roots are $\Delta=\{e_i-e_{i+1},\ 2e_n\mid 1\le i\le n-1\}$, the fundamental weights are $\varpi_i=e_1+e_2+\cdots+e_i$, for $1\le i\le n$.

\subsection*{Type $\mathbf{D}_n$} The positive roots are $\Phi^+=\{e_i\pm e_j \mid 1\le i<j \le n\}$, the simple roots are $\Delta=\{e_i-e_{i+1},\ e_{n-1}+e_n\mid 1\le i\le n-1\}$, the fundamental weights are
\[\begin{cases}
    &\varpi_n = \frac{1}{2}(e_1+\cdots+e_n)\\
    &\varpi_{n-1} = \frac{1}{2}(e_1+\cdots+e_{n-1}-e_n) \\
    &\varpi_i = e_1+e_2+\cdots+e_i, \enspace\text{for $1\le i\le n-2$}.
\end{cases}\]
The minuscule weights are $\varpi_n,\varpi_{n-1},$ and $\varpi_1$.

\subsection*{Type $\mathbf{G}_2$} There are 6 positive roots
\[\Phi^+ = \left\{e_1,\ \sqrt{3}e_2,\ \pm e_1+\frac{\sqrt{3}}{2}e_2,\ \pm\frac{3}{2}e_1+\frac{\sqrt{3}}{2}e_2\right\}.\]
The simple roots and their corresponding fundamental weights are
\[\begin{aligned}[c]
\alpha_1&=e_1 \\
\alpha_2&=-\frac{3}{2}e_1+\frac{\sqrt{3}}{2}e_2
\end{aligned}
\quad\quad\quad
\begin{aligned}[c]
\varpi_1 &= \frac{1}{2}(e_1+\sqrt{3}e_2) \\
\varpi_2 &= \sqrt{3}e_2
\end{aligned}
\]    

\subsection*{Type $\mathbf{F}_4$} There are 24 positive roots 
\[\Phi^+ = \{e_i,\ e_i\pm e_j\mid 1\le i<j\le 4\}\cup\left\{\frac{1}{2}(e_1\pm e_2\pm e_3\pm e_4)\right\}.\]
The simple roots and their corresponding fundamental weights are
\[\begin{aligned}[c]
 &\alpha_1 = e_2-e_3 \\
    &\alpha_2 = e_3-e_4 \\
    &\alpha_3 = e_4 \\
    &\alpha_4 = \frac{1}{2}(e_1-e_2-e_3-e_4)
\end{aligned}
\quad\quad\quad
\begin{aligned}[c]
 &\varpi_1 = e_1+e_2 \\
    &\varpi_2 = 2e_1+e_2+e_3 \\
    &\varpi_3 = \frac{1}{2}(3e_1+e_2+e_3+e_4) \\
    &\varpi_4 = e_1
\end{aligned}
\]

\subsection*{Type $\mathbf{E}_6$} There are 36 positive roots 
\begin{align*}
    \Phi^+ = &\{e_i+e_j\mid1\le i<j\le 5\}\cup\{e_i-e_j\mid1\le j<i\le 5\} \\
    &\cup\left\{\frac{\pm e_1\pm e_2\pm e_3\pm e_4\pm e_5+\sqrt{3}e_6}{2} \ \Big| \ \text{even number of minus signs}\right\}
\end{align*}
The simple roots and corresponding fundamental weights are
\[\begin{aligned}[c]
    &\alpha_1 = \frac{e_1-e_2-e_3-e_4-e_5+\sqrt{3}e_6}{2} \\
    &\alpha_2 = e_1+e_2 \\
    &\alpha_3 = e_2-e_1 \\
    &\alpha_4 = e_3-e_2 \\
    &\alpha_5 = e_4-e_3 \\
    &\alpha_6 = e_5-e_4 
\end{aligned}
\quad\quad
\begin{aligned}[c]
    &\varpi_1 = \frac{2\sqrt{3}}{3}e_6 \\
    &\varpi_2 = \frac{e_1+e_2+e_3+e_4+e_5}{2}+\frac{\sqrt{3}}{2}e_6 \\
    &\varpi_3 = \frac{-e_1+e_2+e_3+e_4+e_5}{2} + \frac{5\sqrt{3}}{6}e_6 \\
    &\varpi_4 = e_3+e_4+e_5+\sqrt{3}e_6 \\
    &\varpi_5 = e_4+e_5+\frac{2\sqrt{3}}{3}e_6 \\
    &\varpi_6 = e_5+\frac{\sqrt{3}}{3}e_6
\end{aligned}
\]

\subsection*{Type $\mathbf{E}_7$} There are 63 positive roots 
\begin{align*}
    \Phi^+ = &\{e_i+e_j\mid1\le i<j\le 6\}\cup\{e_i-e_j\mid1\le j<i\le 6\}\cup\{\sqrt{2}e_7\} \\
    &\cup\left\{\frac{\pm e_1\pm e_2\pm e_3\pm e_4\pm e_5\pm e_6+\sqrt{2}e_7}{2} \ \Big| \ \text{odd number of minus signs}\right\}
\end{align*}
The simple roots are
\begin{align*}
    &\alpha_1 = \frac{e_1-e_2-e_3-e_4-e_5-e_6+\sqrt{2}e_7}{2} \\
    &\alpha_2 = e_1+e_2 \\
    &\alpha_3 = e_2-e_1 \\
    &\alpha_4 = e_3-e_2 \\
    &\alpha_5 = e_4-e_3 \\
    &\alpha_6 = e_5-e_4 \\
    &\alpha_7 = e_6-e_5
\end{align*}
and the corresponding fundamental weights are
\begin{align*}
    &\varpi_1 = \sqrt{2}e_7 \\
    &\varpi_2 = \frac{1}{2}(e_1+e_2+e_3+e_4+e_5+e_6)+\sqrt{2}e_7 \\
    &\varpi_3 = \frac{1}{2}(-e_1+e_2+e_3+e_4+e_5+e_6)+3\frac{\sqrt{2}}{2}e_7 \\
    &\varpi_4 = e_3+e_4+e_5+e_6+2\sqrt{2}e_7 \\
    &\varpi_5 = e_4+e_5+e_6+3\frac{\sqrt{2}}{2}e_7 \\
    &\varpi_6 = e_5+e_6+\sqrt{2}e_7 \\
    &\varpi_7 = e_6+\frac{\sqrt{2}}{2}e_7
\end{align*}

\subsection*{Type $\mathbf{E}_8$} There are 120 positive roots 
\begin{align*}
    \Phi^+ = &\{e_i+e_j\mid1\le i<j\le 8\}\cup\{e_i-e_j\mid1\le j<i\le 8\} \\
    &\cup\left\{\frac{\pm e_1\pm e_2\pm e_3\pm e_4\pm e_5\pm e_6\pm e_7+e_8}{2} \ \Big| \ \text{even number of minus signs}\right\}
\end{align*}
The simple roots are
\begin{align*}
    &\alpha_1 = \frac{e_1-e_2-e_3-e_4-e_5-e_6-e_7+e_8}{2} \\
    &\alpha_2 = e_1+e_2 \\
    &\alpha_3 = e_2-e_1 \\
    &\alpha_4 = e_3-e_2 \\
    &\alpha_5 = e_4-e_3 \\
    &\alpha_6 = e_5-e_4 \\
    &\alpha_7 = e_6-e_5 \\
    &\alpha_8 = e_7-e_6
\end{align*}
and the corresponding fundamental weights are
\begin{align*}
    &\varpi_1 = 2e_8 \\
    &\varpi_2 = \frac{1}{2}(e_1+e_2+e_3+e_4+e_5+e_6+e_7+5e_8) \\
    &\varpi_3 = \frac{1}{2}(-e_1+e_2+e_3+e_4+e_5+e_6+e_7+7e_8) \\
    &\varpi_4 = e_3+e_4+e_5+e_6+e_7+5e_8 \\
    &\varpi_5 = e_4+e_5+e_6+e_7+4e_8 \\
    &\varpi_6 = e_5+e_6+e_7+3e_8 \\
    &\varpi_7 = e_6+e_7+2e_8 \\
    &\varpi_8 = e_7+e_8
\end{align*}


\begin{example}\label{ex: type A}
In type $\mathbf{A}_{n-1}$, for each $1\le d\le n-1$, the maximal parabolic subgroup $P_d$ of $\SL_n$ is determined by $I=\Delta\setminus\{e_d-e_{d+1}\}$, and is associated to the fundamental weight $\varpi_d=e_1+\cdots+e_d$. One deduces 
\begin{align*}
    \Phi_I &= \Z\{e_i-e_{i+1} \mid 1\le i\le n-1,\ \forall i\ne d\} \\
     &= \{e_i-e_j \mid 1\le i<j\le d\}\cup\{e_i-e_j \mid d+1\le i<j\le n\}, \\
    \Phi^+\setminus\Phi_I &= \{e_i-e_j \mid 1\le i\le d, \ d+1\le j\le n\}.
\end{align*}
Equation (\ref{equ: a-invariant root system}) of Corollary \ref{cor: a-invariant via root systems} gives
\[-a(R_{\varpi_d})\varpi_d = \sum_{\alpha\in\Phi^+\setminus\Phi_I}\alpha 
    = \sum_{i=1}^d\sum_{j=d+1}^n(e_i-e_j)= (n-d)\sum_{j=1}^de_j + d\sum_{j=1}^de_j 
    = n\varpi_d,\]
where $-\sum_{j=d+1}^ne_j=\sum_{j=1}^de_j$ follows from the relation $\sum_{i=1}^ne_i=0$ in 
\[X(T)=X(T')\big/\Z(e_1+\cdots+e_n),\]
where $T'$ is a maximal torus in $\GL_n$, so $T=T'\cap\SL_n$ \cite[see p.173]{MR2015057}. It follows that $\fpt(R_{\varpi_d})=-a(R_{\varpi_d})=n$ (notice that this is Proposition \ref{prop: fpt Grassmannian}).
\end{example}

\begin{example}\label{ex: type B}
In type $\mathbf{B}_{n}$, $\varpi_n$ is the only minuscule weight. The corresponding maximal parabolic subgroup $P_n$ of $\SO_{2n+1}$ is determined by $I=\Delta\setminus\{e_n\}$, so 
    \begin{align*}
        &\Phi_I = \Z\{e_i-e_{i+1} \mid 1\le i\le n-1 \} = \{e_i-e_j \mid 1\le i<j\le n\}, \\
        &\Phi^+\setminus\Phi_I = \{e_i,\ e_i+e_j \mid 1\le i<j \le n\}, \\
        &-a(R_{\varpi_n})\varpi_n = \sum_{i=1}^ne_i + \sum_{1\le i<j\le n}(e_i+e_j) = \sum_{i=1}^ne_i + (n-1)\cdot\sum_{i=1}^ne_i = 2n\varpi_n.
    \end{align*}
    Then $a(R_{\varpi_n})=-2n$, and hence $\fpt(R_{\varpi_n})=2n$.
\end{example}

\begin{example}
In type $\mathbf{C}_{n}$, the weight is $\varpi_n$ is not minuscule, so the coordinate ring $R_{\varpi_n}$ of the corresponding homogeneous space $\Sp_{2n}/P_n$ (the Lagrangian Grassmannian) under the Pl\"{u}cker embedding is not an ASL. The maximal parabolic subgroup $P_n$ is determined by $I=\Delta\setminus\{2e_n\}$, so one has
    \begin{align*}
    & \Phi_I = \Z\{e_i-e_{i+1} \mid 1\le i\le n-1\} = \{e_i-e_j \mid 1\le i<j\le n\}, \\
    & \Phi^+\setminus\Phi_I = \{2e_i,\ e_i+e_j \mid 1\le i<j\le n\}, \\
    & -a(R_{\varpi_n})\varpi_n = \sum_{i=1}^n2e_i + \sum_{1\le i<j\le n}(e_i+e_j) = 2\cdot \sum_{i=1}^ne_i + (n-1)\cdot \sum_{i=1}^ne_i=(n+1)\varpi_n.
    \end{align*}
    Then, $\fpt(R_{\varpi_n})=-a(R_{\varpi_n})=n+1$.
\end{example}

\begin{example}
In type $\mathbf{F}_4$, for the simple roots $\alpha_2 = e_3-e_4$ and $\alpha_3 = e_4$ we have:
$\,$
\begin{itemize}
    \item For $I=\Delta\setminus\{\alpha_2\}=\{e_2-e_3, \ e_4, \ \frac{1}{2}(e_1-e_2-e_3-e_4)\}$, we have 
    \[\Phi_I=\left\{e_2-e_3, \ e_4, \ \frac{1}{2}(e_1-e_2-e_3-e_4),\ \frac{1}{2}(e_1-e_2-e_3+e_4)\right\};\]
    \[-a(R_{\varpi_2})\varpi_2 = 7e_1+4e_2+4e_3 + \frac{1}{2}(6e_1+2e_2+2e_3) = \frac{1}{2}(20e_1+10e_2+10e_3) = 5\varpi_2.\]
    Therefore, $\fpt(R_{\varpi_2})=-a(R_{\varpi_2})=5$.
    
    $\,$
    
    \item For $I=\Delta\setminus\{\alpha_3\}=\{e_2-e_3,\ e_3-e_4,\ \frac{1}{2}(e_1-e_2-e_3-e_4)\}$, we have
    \[\Phi_I=\left\{e_2-e_3,\ e_2-e_4,\ e_3-e_4,\ \frac{1}{2}(e_1-e_2-e_3-e_4)\right\};\]
    \[-a(R_{\varpi_3})\varpi_3 = 7e_1+3e_2+3e_3+3e_4 + \frac{1}{2}(7e_1+e_2+e_3+e_4) = 7\varpi_3.\]
    Therefore, $\fpt(R_{\varpi_3})=-a(R_{\varpi_3})=7$. 
\end{itemize}
\end{example}


\begin{remark}
With this approach, one can redo the computations of $\fpt(R_{\varpi_d})$ for the corresponding minuscule weights (as in Examples \ref{ex: type A}, \ref{ex: type B}) and obtain the values in Table \ref{table: minuscule} without the use of principal chains. Using the descriptions of the exceptional type root systems provided above, the same method allows one to compute $\fpt(R_{\varpi_d})$ where $R_{\varpi_d}$ is the coordinate ring of an exceptional type Grassmannian under the Pl\"{u}cker embedding, over an $F$-finite field of characteristic $p$ (note that $R_{\varpi_d}$ is not necessarily an ASL). Below is a complete list of their values:

\begin{table}[H]
\centering 
\caption{The $F$-pure threshold of exceptional type Grassmannians}
\begin{tabular}{ |p{1cm}|p{3cm}|p{3cm}|p{5cm}|  }
\hline
Type & Weight index $d$ & Grassmannian & $\fpt(R_{\varpi_d})$ \\
\hline
$\mathbf{G}_2$ & $1,2$ & $G_2/P_d$ & $5,3$ \\
$\mathbf{F}_4$ & $1,2,3,4$ & $F_4/P_d$  & $6,5,7,8$ \\
$\mathbf{E}_6$ & $1,2,3,4,5,6$ & $E_6/P_d$ & $12,11,9,7,9,12$ \\
$\mathbf{E}_7$ & $1,2,3,4,5,6,7$ & $E_7/P_d$ & $17,14,11,8,10,13,18$ \\
$\mathbf{E}_8$ & $1,2,3,4,5,6,7,8$ & $E_8/P_d$ & 
$23,17,13,9,11,14,19,29$
\\
\hline
\end{tabular}
\label{table: exceptional}
\end{table} 
\end{remark}

As a final note, we have seen that all the formulas for $\fpt(R)$ computed in this paper is independent of the characteristic of the filed. We immediately get the following using Theorem \ref{thm: lct limit of fpt}.

\begin{theorem}
Let $R$ be any one of the coordinate rings of the flag varieties whose $F$-pure threshold we have computed in this section, with underlying field $\C$ the complex numbers. Then $\lct(R)=\fpt(R)$.
\end{theorem}

\begin{proof}
Let $p$ be a prime and $R_p$ be the modulo $p$ reduction of $R$. We have $\fpt(R)=\fpt(R_p)$, so by Theorem \ref{thm: lct limit of fpt}, $\lct(R)=\lim_{p\to\infty}\fpt(R_p)=\fpt(R)$.
\end{proof}

$\,$

$\,$

\bibliographystyle{alpha} 
\bibliography{bibliography.bib}

\end{document}